\documentclass[preprint,12pt]{elsarticle}

\biboptions{comma}

\usepackage{amsmath}
\usepackage{amsfonts}
\usepackage{amssymb}
\usepackage{amsthm}
\usepackage{empheq}

\usepackage[backgroundcolor=gray!30,linecolor=black]{todonotes}
\theoremstyle{plain}
\newtheorem{theorem}{Theorem}[section]

\theoremstyle{definition}

\theoremstyle{remark}
\newtheorem{remark}[theorem]{Remark}

\numberwithin{equation}{section} %% Equation numbering control.
\numberwithin{figure}{section}   %% Figure numbering control.

\newcommand{\vect}[1]{\mathbf{#1}}

\newcommand{\bu}{\vect{u}}

\newcommand{\bx}{\vect{x}}

\newcommand{\bomega}{\boldsymbol{\omega}}

\newcommand{\field}[1]{\mathbb{#1}}

\newcommand{\nZ}{\field{Z}}

\newcommand{\nR}{\field{R}}

\newcommand{\maps}{\rightarrow}

\begin{document}

% \journal{Phys. Lett. A}
% Highlights:
% A blow-up criterion for the 3D Euler incompressible equations is proposed.
% It is based on the 3D Euler-Voigt inviscid regularization.  
% It is stronger than a criterion proposed in a previous work by the authors.
% It is also better adapted for computational tests.

%==============================================================-
\begin{frontmatter}
\title{A Blow-Up Criterion for the 3D Euler Equations Via the  Euler-Voigt Inviscid Regularization}
\date{July 27, 2015}

\author[unl]{Adam Larios\corref{cor1}}
\address[unl]{Department of Mathematics,
                University of Nebraska--Lincoln,
                203 Avery Hall,
        Lincoln, NE 68588--0130, USA}
\ead{alarios@unl.edu}
\cortext[cor1]{Corresponding author}

\author[tamu]{Edriss~S.~Titi}
\address[tamu]{Department of Mathematics,
                Texas A\&M University,
                3368 TAMU,
        College Station, TX 77843--3368, USA. Also,
        Department of Computer Science and Applied Mathematics,
Weizmann Institute of Science,
Rehovot 76100, Israel.
        }
\ead{titi@math.tamu.edu, edriss.titi@weizmann.ac.il}

% \maketitle
\begin{abstract}
We propose a new blow-up criterion for the 3D Euler equations of incompressible fluid flows, based on the 3D Euler-Voigt inviscid regularization.  This criterion is similar in character to a criterion proposed in a previous work by the authors, but it is stronger, and better adapted for computational tests.   The 3D Euler-Voigt equations enjoy global well-posedness, and moreover are more tractable to simulate than the 3D Euler equations.  A major advantage of these new criteria is that one only needs to simulate the 3D Euler-Voigt, and not the 3D Euler equations, to test the blow-up criteria, for the 3D Euler equations, computationally.
% Moreover, it is a potentially stronger criterion, in the sense that the set of singularities it can detect is a (possibly proper) superset of the singularities detectable by the previous criterion.
\end{abstract}

\begin{keyword}
Euler-Voigt \sep Navier-Stokes-Voigt \sep Inviscid regularization \sep Turbulence models \sep $\alpha-$Models \sep Blow-up criterion for Euler.

\MSC
35B44 \sep % Blow-up
35A01 \sep % Existence problems: global existence, local existence, non-existence
35Q30 \sep % Navier--Stokes equations
35Q31 \sep % Euler equations
35Q35 \sep % PDEs in connection with fluid mechanics
76B03 % Existence, uniqueness, and regularity thoery (Incompressible inviscid fluids)
%
% 35-04, % Explicit machine computation and programs (not the theory of computation or programming)
% 35A01, % Existence problems: global existence, local existence, non-existence
% 35A02, % Uniqueness problems: global uniqueness, local uniqueness, non-uniqueness
% 35B30, % Dependence of solutions on initial and boundary data, parameters
% 35B41, % Attractors (PDEs)
% 35B44, % Blow-up
% 35B65, % Smoothness and regularity of solutions
% 35K51, % Initial--boundary value problems for second-order parabolic systems
% 35K55, % Nonlinear parabolic equations
% 35Q30, % Navier--Stokes equations
% 35Q31, % Euler equations
% 35Q35, % PDEs in connection with fluid mechanics
% 35Q86, % PDEs in connection with geophysics
% 37L30, % Attractors and their dimensions, Lyapunov exponents
% 65N99, % General Numerical Analysis of BVPs for PDEs
% 76A10, % Viscoelastic fluids
% 76B03, % Existence, uniqueness, and regularity theory (Incompressible inviscid fluids)
% 76D03, % Existence, uniqueness, and regularity theory (Incompressible viscous fluids)
% 76D09, % Viscous-inviscid interaction
% 76F05, % Isotropic turbulence; homogeneous turbulence
% 76F20, % Dynamical systems approach to turbulence [See also 37-XX]
% 76F25, % Turbulent transport, mixing
% 76F55, % Statistical turbulence modeling [See also 76M35]
% 76F65, % Direct numerical and large eddy simulation of turbulence
% 76W05, % Magnetohydrodynamics and electrohydrodynamics
\end{keyword}

\end{frontmatter}

% \maketitle
\thispagestyle{empty}%Gets rid of page number on first page.
%============================================================

% =====================================================================
\section{Introduction}\label{sec:Int}
% =====================================================================
\noindent

A major difficulty in the computational search for blow-up of the 3D incompressible Euler equations is that one must seemingly simulate the 3D Euler equations themselves to obtain information about  singularities.  Near the time of a potential singularity, sufficient accuracy of simulations of the 3D Euler equations can be challenging
% ---indeed, controversial---
to obtain, due to the need to resolve spatial derivatives that are potentially infinite.  However, two new blow-up criteria, one proved by the authors in \cite{Larios_Titi_2009},  and another provided in the present work, provide a path around this difficulty by using the Euler-Voigt inviscid regularization.  In particular, by tracking the $L^2$-norm of the vorticity of solutions to the 3D Euler-Voigt equations, as a certain regularizing parameter $\alpha$ tends to zero, these new criteria allow one to gather evidence for potential singularities of the 3D Euler equations by only simulating the 3D Euler-Voigt equations.  This is advantageous from a computational standpoint, because the $L^2$-norm of the spatial gradient of solutions to the Euler-Voigt equations is uniformly bounded in time, for any fixed value of the regularization parameter $\alpha$.  Furthermore, all higher-order norms grow at most algebraically in time \cite{Larios_Titi_2009}, which implies that pointwise spatial derivatives grow at most algebraically in time.  No such results are known for the 3D Euler equations.  This means that simulating the 3D Euler-Voigt equations is  computationally more tractable than simulating the 3D Euler equations, since achieving sufficient accuracy requires that simulations have high enough resolution to resolve spatial gradients.

In this work, we provide a new blow-up criterion that is similar in character to the criterion in \cite{Larios_Titi_2009}, but that has several advantages over the previous criterion.  Our main focus is on differences in the computational implementation of the two criteria.  However, one analytical advantage is that the new criterion is potentially\footnote{Of course, it may be that there are no singularities in the 3D Euler equations.} stronger than the previous criterion.  This is because the set of singularities it can detect is a (possibly proper) superset of the singularities detectable by the criterion in \cite{Larios_Titi_2009}.
From the standpoint of computational implementation, the criterion in \cite{Larios_Titi_2009} does not allow for simulations with adaptive or variable time-stepping, requiring a fixed time step in each simulation, and also agreement in time-steps across all simulations as the regularization parameter varies. (Interpolating in time is also an option, but introduces additional approximation.)  The new criterion only requires that the simulations end at a common final time.
Furthermore, the previous criterion requires computation and output of the $L^2$-norm of the velocity gradient (or the vorticity) at every time step, which can be costly, requiring additional memory storage, and---in parallel simulations---additional communications.  The new criterion requires this data only at the final time.
We note that both of these criteria are only known to be sufficient for blow-up; they are not known to be necessary for blow-up, unlike, e.g., the Beale-Kato-Majda criterion \cite{Beale_Kato_Majda_1984}.
Currently, both criteria are being tested computationally \cite{Larios_Petersen_Titi_Wingate_2015}.

The computational search for blow-up of the 3D incompressible Euler equations has a long history (see, e.g., \cite{Deng_Hou_Yu_2005,Gibbon_2008,Hou_2009,Hou_Li_2008_Blowup,Hou_Li_2008_Numerical,Kerr_1993}, and the references therein).  Traditionally, one attempts to identify singularities by means of blow-up criteria based on quantities arising from the 3D Euler equations.  There are many such criteria in the literature (see, e.g.,  \cite{Beale_Kato_Majda_1984,Constantin_Fefferman_1993,Constantin_Fefferman_Majda_1996,Ferrari_1993,Gibbon_Titi_2013_Blowup,Ponce_1985}, and the references therein).  Computational tests of the type of blow-up criteria described here and in \cite{Larios_Titi_2009} require more simulations (of the 3D Euler-Voigt equations) than tests requiring a single simulation of the 3D Euler equations, since one must run simulations for several values of the regularizing parameter $\alpha$.  However, these simulations require less resolution than simulations of the Euler equations as discussed above.
% Moreover, in the search for blow-up, accuracy is of more interest than computational speed.
% Therefore, we do not consider the need to run additional simulations to cover a range of $\alpha$-values to be of great concern.

% We note that for both criteria, data must be gathered and stored for every simulation (corresponding to different $\alpha$-values) before a single contour can be formed.  This issue appears unlikely to be overcome, as the dependence of solutions on $\alpha$ is highly non-trivial.  However, this is not very demanding in terms of memory storage, since only a single number must be stored for each $\alpha$-value and each time step.

The Euler-Voigt inviscid regularization of the Euler equations is given by
\begin{subequations}\label{EV}
\begin{empheq}[left=\empheqlbrace]{align}
\label{EV_mo}
-\alpha^2\partial_t\nabla^2\bu
+\partial_t\bu +(\bu\cdot\nabla)\bu
+\nabla p&=0%\nu\nabla^2\bu +\bF
,
\\
\label{EV_div}
\nabla\cdot \bu
&= 0,
\\
\label{EV_int}
\bu(\bx,0) &= \bu_0(\bx).
\end{empheq}
\end{subequations}
The parameter $\alpha>0$, having units of length, is the regularizing parameter.   Formally, setting $\alpha=0$, we recover the incompressible Euler equations.  The fluid velocity field, $\bu=\bu(\bx,t)$, and the fluid pressure, $p=p(\bx,t)$ are the unknown quantities. We consider system \eqref{EV} in a periodic box $\Omega:=\nR^3/\nZ^3=\equiv[0,1]^3$.  We assume that the spatial average of $\int_\Omega\bu_0(\bx)\,d\bx=0$.  With \eqref{EV_mo}, this implies $\int_\Omega\bu(\bx,t)\,d\bx=0$ for all $t$.  From now on, we denote by $\bu^\alpha$ the solution to \eqref{EV}, and by $\bu$ a  solution to the Euler equations (i.e., \eqref{EV} with $\alpha=0$), both starting from the same sufficiently smooth initial condition $\bu_0$.
We denote the vorticity $\bomega := \nabla\times\bu$, and $\bomega^\alpha := \nabla\times\bu^\alpha$.

% \todo[inline]{Make it sounds more interesting and not just a listing of facts.}
The Euler-Voigt equations were first described and analyzied in \cite{Cao_Lunasin_Titi_2006}, where they were shown to be globally well-posed for all initial data $\bu_0\in H^1$ and all $t>0$.  We note that their viscous counterpart, called the Navier-Stokes-Voigt equations, were proposed and studied much earlier in \cite{Oskolkov_1973, Oskolkov_1982}, as a model for Kelvin-Voigt viscoelastic fluids.  The Euler-Voigt equations have been studied computationally in \cite{DiMolfetta_Krstlulovic_Brachet_2015,Larios_Petersen_Titi_Wingate_2015}.  The Euler-Voigt and Navier-Stokes-Voigt equations, along with extensions of these models, have been studied in both analytical and numerical contexts (see, e.g., \cite{Bohm_1992,Cao_Lunasin_Titi_2006,Catania_2009,Catania_Secchi_2009,DiMolfetta_Krstlulovic_Brachet_2015,Ebrahimi_Holst_Lunasin_2012,Kalantarov_Levant_Titi_2009,Kalantarov_Titi_2009,Khouider_Titi_2008,Kuberry_Larios_Rebholz_Wilson_2012,Larios_Lunasin_Titi_2015,Larios_Titi_2009,Levant_Ramos_Titi_2009,Olson_Titi_2007,Oskolkov_1973,Oskolkov_1982,Ramos_Titi_2010}).  

The following theorem was proved in \cite{Larios_Titi_2009} (see also  a similar theorem for the surface quasi-geostrophic (SQG) equations in \cite{Khouider_Titi_2008}).
\begin{theorem}[\cite{Larios_Titi_2009}]\label{thm_blow_up_old}
Assume $\bu_0\in H^s$, for some $s\geq3$, with $\nabla\cdot\bu_0=0$. Suppose there exists a $T^{*} >0$ such that the solutions $\bu^\alpha$ of \eqref{EV}, with initial data $\bu_0$,  satisfy
\begin{align}\label{old_blow_up}
   \sup_{t\in[0,T^{*}]}\limsup_{\alpha\maps0^+}\left(\alpha\|\nabla \bu^\alpha(t)\|_{L^2}\right)>0.
\end{align}
Then the 3D Euler equations, with initial data $\bu_0$, develop a singularity within the interval $[0,T^{*}]$.
\end{theorem}

\begin{remark}
Since $\nabla\cdot\bu^\alpha=0$, integration by parts can be used to show that 
\begin{align}
   \|\nabla\bu^\alpha(t)\|_{L^2}
   \equiv
   \|\bomega^\alpha(t)\|_{L^2}.
\end{align}
Therefore, in line with other blow-up criteria in the literature, \eqref{old_blow_up} can be seen as a condition on the vorticity, albeit from the 3D Euler-Voigt equations rather than the 3D Euler equations.
\end{remark}

A technical difficulty arises in computational tests of Theorem \ref{thm_blow_up_old}. Mathematically, one may imagine fixing a $t>0$ and computing
\begin{align}\label{blow_up_quantity_1}
\limsup_{\alpha\maps0^+}\left(\alpha\|\nabla \bu^\alpha(t)\|_{L^2}\right).
\end{align}
However, computationally, it is more natural to first fix an $\alpha>0$,  as a parameter, and then to compute $\bu^\alpha(t)$ as $t$ increases up to a time $T$ (e.g., by a standard time-stepping method).  Therefore, to construct curves of $\alpha$ vs. $\alpha\|\nabla \bu^\alpha(t)\|_{L^2}$ for each fixed $t$, one must jump from solution to solution as $\alpha$ varies.  This gives rise to some of the technical issues discussed above. However, suppose for a moment that one is allowed to commute the two limiting operations in \eqref{old_blow_up}.  In this case, one would then be interested whether
\begin{align}\label{blow_up_max_first}
\limsup_{\alpha\maps0^+}\Big(\alpha\sup_{t\in[0,T^{*}]}\|\nabla \bu^\alpha(t)\|_{L^2}\Big)>0.
\end{align}
The quantity in \eqref{blow_up_max_first} is arguably easier to track, as discussed above.
% : one only needs the final output time $T^*$ to agree, and variable time-stepping can be used for individual simulations.  On the other hand, even if \eqref{blow_up_max_first} indicates the formation of a finite-time singularity in the 3D Euler equations, it only indicates that the singularity occurs within the interval $[0,T^*]$, whereas tracking curves at the same time-steps as described above in analyzing \eqref{old_blow_up} may give more detailed information.
It is the purpose of this work to show rigorously that \eqref{blow_up_max_first} implies that the 3D Euler equations develop a singularity within the interval $[0,T^*]$.

% =====================================================================
\section{Notation and Preliminary Results}\label{sec:pre}
% =====================================================================
We denote by $L^p$ and $H^s$ the usual Lebesgue and Sobolev spaces over the periodic domain $\Omega\equiv[0,1]^3:=\nR^3/\nZ^3$, respectively.  It is a classical result (see, e.g., \cite{Majda_Bertozzi_2002,Marchioro_Pulvirenti_1994}) that, for initial data $\bu_0\in H^3$ satisfying $\nabla\cdot\bu_0=0$, a unique strong solution $\bu$ of the 3D Euler equations exists and is unique on a maximal time interval that we denote by $[0,T^*)$.  Moreover, one has
\begin{align}\label{Euler_energy_equality}
\|\bu(t)\|_{L^2} = \|\bu_0\|_{L^2} \text{ on } [0,T^*).
\end{align}
Equation \eqref{Euler_energy_equality} holds under weaker conditions on the smoothness of the solutions of the 3D Euler equations, as it was conjectured by Onsager (see, e.g., \cite{Cheskidov_Constantin_Friedlander_Shvydkoy,Constantin_E_Titi_1994,Eyink_1994,Onsager_1949}).  However, the existence of such weak solutions for arbitrary admissible initial data is still out of reach.  In \cite{Bardos_Titi_2010}, it was shown that a certain class of shear flows are weak solutions in $L^\infty((0,T);L^2)$ that conserve energy.  Furthermore, families of weak solutions that do not satisfy the regularity assumed in the Onsager conjecture have been constructed that do not satisfy \eqref{EV_energy_equality} \cite{Buckmaster_2015,DeLellis_Szekelyhidi_2010_admissibility,DeLellis_Szekelyhidi_2014_Onsager,Isett_2013_thesis}.%\todo{Are these the right references about the dissipation anomoly?  Check carefully.}
 The following ``$\alpha$-energy equality'' was proven in \cite{Cao_Lunasin_Titi_2006}.
\begin{theorem}\label{thm_energy_equality}
Let $\bu_0\in H^1$ with $\nabla\cdot\bu_0=0$, and let $\bu^\alpha$ be the corresponding solution to \eqref{EV}.  Then, for any $t\in\nR$,
\begin{align}\label{EV_energy_equality}
   \|\bu^\alpha(t)\|_{L^2}^2 +\alpha^2\|\nabla\bu^\alpha(t)\|_{L^2}^2
   %+2\nu\int_0^T\|\nabla\bu(s)\|_{L^2}^2\,ds
   %\\&
   =
   \|\bu_0\|_{L^2}^2+\alpha^2\|\nabla\bu_0\|_{L^2}^2.
\end{align}
\end{theorem}

The following convergence theorem was proven in \cite{Larios_Titi_2009}.
\begin{theorem}\label{convergence}
Let $\bu_0\in H^s$, $s\geq3$ with $\nabla\cdot\bu_0=0$, and let $[0,T^*)$ be the corresponding  maximal interval  of existence and uniqueness of the solution, $\bu$, to the 3D Euler equations.  Choose $T\in[0,T^*)$.  Then there exists a constant $C>0$, which depends on $\sup_{0\le t \le T} \|u(t)\|_{H^3}$, such that for all $t\in[0,T]$,
\begin{align}\label{conv_claim_full}
\|\bu(t)-\bu^\alpha(t)\|_{L^2}^2
+\alpha^2\|\nabla(\bu(t)- \bu^\alpha(t))\|_{L^2}^2
\leq
C\alpha^2(e^{Ct}-1).
\end{align}
\end{theorem}

% =====================================================================
\section{An Improved Blow-up Criterion}\label{sec:blowup}
% =====================================================================
Let $T>0$ be given.  Assume that a given solution to the Euler equations is smooth on $[0,T]$, so that in particular, \eqref{Euler_energy_equality} holds.  We emphasize that \eqref{Euler_energy_equality} depends on the regularity of the Euler equations, and if a finite-time singularity develops, \eqref{Euler_energy_equality} might not hold.

\begin{theorem}
Let $\bu_0\in H^s$, $s\geq3$, with $\nabla\cdot\bu_0=0$, and let $\bu^\alpha$ be the corresponding unique solution of \eqref{EV}.  Suppose that
 \begin{align}\label{blowup_new}
 \limsup_{\alpha\maps0^+}\sup_{t\in [0,T]} \alpha\|\nabla\bu^\alpha(t)\|_{L^2}>0,
\end{align}
for some $T>0$.  Then the unique solution to the 3D Euler equations, with initial data $\bu_0$,  must develop a singularity within the interval $[0,T]$.
\end{theorem}

\begin{proof}
We prove the contrapositive.  Assume that $\bu$ is a solution of the 3D Euler equations, with initial data $\bu_0\in H^s$, $s\geq3$, that remains smooth on the interval $[0,T]$.  In particular, the smoothness implies that \eqref{Euler_energy_equality} holds.  From \eqref{conv_claim_full} there exists a constant $C>0$, depending on $\sup_{0\le t \le T} \|u(t)\|_{H^3}$, such that
\begin{align}\label{conv_claim}
\|\bu^\alpha(t)\|_{L^2}
&\geq
\|\bu(t)\|_{L^2}-C\alpha(e^{Ct}-1)^{1/2}
\geq
\|\bu(t)\|_{L^2}-C\alpha(e^{CT}-1)^{1/2}
\\&=\nonumber
\|\bu_0\|_{L^2}-C\alpha(e^{CT}-1)^{1/2}.
\end{align}
Here, we have used \eqref{Euler_energy_equality}.
Let $\alpha>0$ be small enough that the right-hand side is positive (i.e., $\alpha< \|\bu_0\|_{L^2}/(C(e^{CT}-1)^{1/2})$.  Squaring, we obtain
\begin{align}\label{L2est}
\|\bu^\alpha(t)\|_{L^2}^2
\geq
\|\bu_0\|_{L^2}^2
-2C\alpha\|\bu_0\|_{L^2}(e^{CT}-1)^{1/2}
+C^2\alpha^2(e^{CT}-1),
\end{align}
for every $t\in[0,T]$.
Combining \eqref{L2est} and \eqref{EV_energy_equality}, we discover
 \begin{align*}
\alpha^2\|\nabla\bu^\alpha(t)\|_{L^2}^2
\leq
\alpha^2\|\nabla\bu_0\|_{L^2}^2
+2C\alpha\|\bu_0\|_{L^2}(e^{CT}-1)^{1/2}
-C^2\alpha^2(e^{CT}-1).
\end{align*}
Thus,
$
 \limsup_{\alpha\maps0^+}\sup_{t\in [0,T]} \alpha^2\|\nabla\bu^\alpha(t)\|_{L^2}^2=0,
$
which contradicts assumption \eqref{blowup_new}, and therefore the solution $\bu$, of the 3D Euler equations, is singular within the interval $[0,T]$.
\end{proof}

% =====================================================================
\subsection{Comparison with original criterion}
% =====================================================================
\noindent
We show that the new blow-up criterion \eqref{blowup_new}  is stronger than  \eqref{thm_blow_up_old}.
% , as it has the potential to detect singularities that  \eqref{thm_blow_up_old} cannot.
Since
\begin{align}
\sup_{t \in [0,T]} \alpha^2 \| \nabla u (t) \|_{L^2}^2 \ge \alpha^2 \| \nabla u (t) \|_{L^2}^2,
\end{align}
for any $t\in[0,T]$, we may take the $\limsup_{\alpha \to 0^+}$ of both sides to obtain
\begin{align}
\limsup_{\alpha \to 0^+} \sup_{t \in [0,T]} \alpha^2 \| \nabla u (t) \|_{L^2}^2 \ge \limsup_{\alpha \to 0^+} \alpha^2 \| \nabla u (t) \|_{L^2}^2.
\end{align}
The left-hand side is constant, and the right-hand side depends on t. Thus,
\begin{align}
\limsup_{\alpha \to 0^+} \sup_{t \in [0,T]} \alpha^2 \| \nabla u (t) \|_{L^2}^2 \ge  \sup_{t \in [0,T]} \limsup_{\alpha \to 0^+} \alpha^2 \| \nabla u (t) \|_{L^2}^2.
\end{align}
Therefore, if the right-hand side is positive, the left-hand side is positive.  Hence, \eqref{old_blow_up} implies \eqref{blowup_new}.
% However, the left-hand side could be positive while the right-hand side is zero.

% =====================================================================
\section*{Acknowledgments}
% =====================================================================
The work of E.S.T was supported in part by ONR grant number N00014-15-1-2333, and by the NSF grants number DMS-1109640 and DMS-1109645.

% =====================================================================
\section*{Bibliography}
% =====================================================================

%~~~~~~~~~~~~~~~~~~~~~~~~~~~~~~~~~~~~~~~~~~~~~~~~~~~~~~~~~~~~~~~~~~~~
% \begin{scripstize}
\bibliographystyle{abbrv}%amsalpha%amsplain%plain
\def\cprime{$'$}

%~~~~~~~~~~~~~~~~~~~~~~~~~~~~~~~~~~~~~~~~~~~~~~~~~~~~~~~~~~~~~~~~~~~~
\end{document}